\documentclass[a4paper,10pt]{article}
\usepackage{graphicx}

\usepackage{geometry}

\usepackage[ngerman]{babel}
\usepackage{amssymb,amsmath}
\usepackage[usenames,dvipsnames]{color}
\usepackage{amsfonts}
\usepackage{tikz}
\usepackage{hyperref}
\hypersetup{colorlinks,citecolor=black,filecolor=black,linkcolor=black,urlcolor=black}

\usepackage{framed}

\usepackage{latexsym}
\usepackage{amsthm}
\usepackage{bm}

\usepackage{amscd}
\usepackage[all]{xy}

\renewcommand{\abstract}{\textbf{Abstract}}
\usepackage{cite}
\newcommand{\I}{\mathrm{i}}
\renewcommand{\d}{{\rm d}}

\newtheorem{satz}{Satz}[subsection]
\newtheorem{theorem}[satz]{\textit{Theorem}}

\newtheorem{proposition}[satz]{\textit{Proposition}}
\newtheorem{korollar}[satz]{\textit{Corollary}}
\newtheorem{definition}[satz]{\textit{Definition}}
\numberwithin{equation}{section}
\date{}

\vspace{-0.4cm}\title{\textbf{\textit{Some notes on Euler products}}\vspace{-0.4cm}}
\author{\textit{Johannes L\"offler}}
\begin{document}
\maketitle
\begin{abstract} 
We focus on a well-known convergence phenomenon, the fact that the $\zeta$ zeros are the universal singularities of certain Euler products. \end{abstract}

\subsubsection*{\textit{Introduction to the Riemann zeta $\zeta$ function}}
The most interesting aspects of the $\zeta$ function may be its connection to prime numbers and its in some sense chaotic universality discovered by Voronin \cite{Vo} relying on highly non-trivial convergence analysis. The connection of $\zeta$ to number theory is due to Euler's famous formula
\begin{align}\label{ZEP}
\zeta(s):=\prod_{p\in\mathbb{P}}\frac{1}{1-p^{-s}}=\sum_{n=1}^{\infty}n^{-s}
\end{align}
where $p$ runs through the set $\mathbb{P}$ of primes. Here we used the fundamental theorem of arithmetic to convert the product over primes of geometric series into a sum over the positive integers: In words the $\zeta$ function encodes the connection of the natural numbers to the multiplication atoms known as primes. The previous formula \ref{ZEP} only holds for $s$ with real part $>1$, but the analytic continuation of $\zeta$ in a natural way gives sense to diverging expressions like Ramanujan's summation
$1+2+3+4+5+\dots=\zeta(-1)=-1/12$
or the also very weird seeming, but useful summation
$1+1+1+1+1+\dots=\zeta(0)=-1/2$
informally corresponding to the trace of the identity matrix in infinite dimensions, a regularization used in field theories.

The $\zeta$ and multiple $\zeta$ values appear in quantization procedures and other topics, the article \cite{ZA} contains a summary of some interesting aspects of $\zeta$ functions and special values.

Riemann proved \cite{RI} that the continuation of $\zeta$ has so-called trivial zeros at the negative even integers $-2,-4,-6,-8,-10,\cdots$ and conjectured  that all non-trivial zeros of $\zeta$ have real part ${1}/{2}$. Maybe the main evidences for him are the functional $\zeta(1-s)=\zeta(s)2\cos\left(\frac{\pi{s}}{2}\right)\Gamma(s)(2\pi)^{-s}$, his numerical calculations and the fact that the functions $1/(1-p^{-s})$ have singularities at the quite wild distributed points ${s=\I2\pi{n}/\ln(p),\;n\in\mathbb{Z}}$
hence all the singularities of all factors of the Euler product \ref{ZEP} are on the line $\I\mathbb{R}$. The intuition of the conjecture is that at least the geometric structure of the zeros remains a vertical line, but by analytic continuation the line gets somehow shifted by $1/2$ to the right and the previous complicated distribution
modified by an unknown mechanism. Knauf \cite{K} reasoned on a connection to a famous theorem of Lee and Yang \cite{LeeYang} in statistical physics where certain partition functions of ferromagnetic systems admit only real zeros. There are many other statements that have been proven to be indeed equivalent to Riemann's original conjecture and it has nice consequences as a hypothesis, called RH, including the best possible bound for the error of the prime number theorem \cite{Ko}. It is widely hoped that a proof of Riemann's conjecture should with some technical difficulties at the end yield a proof of the generalized Riemann conjecture for a certain class of functions called $L$-functions. We list some of the major results, speculations on RH:
\begin{itemize} 
\item A classical result of Hadamard \cite{HA} and Poussin \cite{POU}, based on
$0\leq3+4\cos(\alpha)+\cos(2\alpha)\forall\alpha\in\mathbb{R}$
contradicts zeros on the boundary $\Re(s)=\pm1$ of the critical strip.
\item Hardy and Littlewood \cite{Ha} proved that there are infinitely many non trivial zeros on the critical line, for an alternative proof see \cite{PO}. Bohr and Landau showed \cite{BL} that in some sense most of the non trivial zeros are located to the critical line. Conrey significantly improved old results and showed \cite{Con} that a positive percentage greater than $2/5$ of all zeros is on the critical line.
\item Choie, Lichiardopol, Sole and Moree proved in \cite{Moree} that certain conditions of an equivalent, but in some sense more elementary statement called Robin's criteria \cite{Robin} are satisfied.
\item Zagier showed that the existence of a certain inner product would imply the Riemann conjecture, we refer the reader for the details to the article \cite{ZA2}.
\end{itemize}

\textbf{Acknowledgements:} I am deeply grateful to P. Moree for his interest in my research, useful comments and discussions. I thank S. Bhattacharya, G. Bhowmik, D. Essouabri, H. Furusho, O. Ramar\'e and S. Saad-Eddin for reading draft versions, discussions of my notes and pointing out references. Last but not least I thank the MPIM \textsc{Max-Planck-Institute For Mathematics} in Bonn. johannes@mpim-bonn.mpg.de

\section{\textit{Euler products and universal zeros}}\label{NonZero}
The idea in the following is to cancel the problematic divergent terms to get a converging Euler product, we do not claim originality and just want to give the reader a simple introduction in the topic. The statements here have an analogy to results of Estermann, Dahlquist, Kurokawa and for example the articles \cite{BMa} and \cite{BEL} are interesting for readers who want to go further. Theorem \ref{Uni} implies that the $\zeta$ zeros and singularities are the universal singularities of a certain class of Euler products, because it allows to construct analytic continuations for certain Euler products, like the well-known example
$\prod_{p}{1}/({1+p^{-s}})=\zeta(2s)/\zeta(s)$
based on the binomial theorem
$(\alpha+\beta)(\alpha-\beta)=\alpha^2-\beta^2$
or the corollary \ref{Universality}.
In this example $\zeta(2s)$ is the converging Euler product. It might be worth to mention that $\prod_{p}{1}/({1+p^{-s}})$ at integers relates the well-known $\zeta$ values at even integers with the mysterious $\zeta$ values at the odd integers. The theorem \ref{Uni} is folklore and should be in some generalized version well-known
\begin{theorem}\label{Uni}
Consider a sequence $g_n$ of on the unit disc $\mathbb{D}_1(0)$ holomorphic functions with $g_n(x)\neq1$ for $x\in\mathbb{D}_1(0)$ and with the equicontinuity property that there exists $C,\delta,\lambda>0$ so that
$\vert{x}\vert\leq\delta\Rightarrow\vert{g_n}(x)\vert\leq{C}\vert{x}\vert^\lambda$.
Let $\{a_n\}_{n\in\mathbb{N}^+}$ be a sequence with $a_{n+1}>a_n\in\mathbb{N}^+$. The Euler product
\begin{equation}\label{BIGPRO}
\prod_{n=1}^{\infty}\frac{1}{1-g_n(a_n^{-s})}
\end{equation}
converges absolutely for $s\in\mathbb{C}$ with $\Re(s)>1/\lambda$ to a value $\in\mathbb{C}\setminus\{0,\infty\}$.
\end{theorem}
\begin{proof}[Proof] It is standard that \ref{BIGPRO} converges absolutely if $\sum_{n\in\mathbb{N}}\ln\bigr(1-g_n(a_n^{-s})\bigr)$ converges absolute. Clearly a finite number of non-singular values does not influence convergence and we can argue as usual: The well-known series representation $\ln(1-x)=-\sum_{n=1}^{\infty}x^n/n$ and for small $x$ the estimation
$\vert{x}\vert/2<\vert\ln(1-x)\vert<3\vert{x}\vert/2$
holds. Now the equicontinuity assumption that there exists some real numbers $C,\delta,\lambda>0$ so that $\vert{x}\vert\leq\delta\Rightarrow\vert{g_n}(x)\vert\leq{C}\vert{x}\vert^\lambda$ combined with the convergence of
$\sum_{n\in\mathbb{N}}1/n^{1+\epsilon}=\zeta(1+\epsilon)
$
for all $\epsilon\in\mathbb{R}^+$
imply that
$\sum_{n\in\mathbb{N}}\big|\ln\bigr(1-g_n(a_n^{-s})\bigr)\big|$
converges absolutely for $\Re(s)>1/\lambda$. In fact we do not have to restrict the numbers $a_n$ in the sequence $\{a_n\}_{n\in\mathbb{N}^+}$ to be integers, \ref{BIGPRO} also converges if we suppose $a_{n+1}>a_n\geq1\in\mathbb{R}$ and absolute convergence of
$\sum_{n\in\mathbb{N}}1/\vert{a_n}\vert^{1+\epsilon}$.
Under this assumptions the inverse of \ref{BIGPRO} also converges if we drop the restriction $g_n(x)\neq1$ but in this case possibly to zero, we thank O. Ramar\'e and P. Moore for pointing this out.\end{proof}
\begin{korollar}\label{Universality}
Let $\{\mathbf{l}_n\}_{n\in\mathbb{N}}$ be a sequence with $\mathbf{l}_n\in[-1,1]$, $\{a_n\}_{n\in\mathbb{N}^+}$ be a sequence with $a_{n+1}>a_n\in\mathbb{N}^+$ and
$f(x)=1-\lambda_1x-\sum_{i=m}^{\infty}{\lambda_i}x^i$
with $2\leq{m}\in\mathbb{N}$ a holomorphic power series that converges for $x\in\mathbb{D}_1(0)$ to a non zero value. The Euler product
$$\prod_{n=1}^{\infty}\frac{f\bigr(\mathbf{l}_na_n^{-s}\bigr)}{1-\lambda_1\mathbf{l}_na_n^{-s}}$$
converges absolutely for $s\in\mathbb{C}$ with $\Re(s)>1/m$ to a value $\in\mathbb{C}\setminus\{0,\infty\}$. Hence if $\Re(s)>1/2$ and $f(x)=1-x-\sum_{i=m}^{\infty}{\lambda_i}x^i$ then $\zeta(z)=0$ implies $\prod_{p\in\mathbb{P}}1/f\left(p^{-s}\right)=0$ with the same multiplicity.\end{korollar}
\begin{proof}[Proof] Let us restrict to $l_p=1$, there are no essential changes in the more general case: First 
$\frac{f(x)}{1-\lambda_1{x}}=1-\frac{\sum_{i=m}^{\infty}{\lambda_i}x^{i}}{1-\lambda_1{x}}$
is a Taylor series starting with $1$ and with in some sense lowest degree $m$
and by the assumptions $\Re(s)>0$ and $f$ being holomorphic we have that if $p$ is large enough we can find $\epsilon>0$ to estimate
$\Big|\frac{\sum_{i=m}^{\infty}{\lambda_i}a_n^{-is}}{1-\lambda_1a_n^{-s}}\Big|\leq\vert{a_n}^{-m\Re{(s)}}\vert\vert\lambda_m+\epsilon\vert$
Now we can just use the previous theorem \ref{Uni} to conclude corollary \ref{Universality}.\end{proof}
\begin{definition}
Let $\{\mathbf{l}_n\}_{n\in\mathbb{N}}$ be a sequence with $\mathbf{l}_n\in[-1,1]$ and $\mathbb{A}$ be a subset $\mathbb{A}=\{a_n\}_{n\in\mathbb{N}}$ of $\mathbb{N}^+$ with $a_{n+1}>a_n$. We define the associated Dirichlet series $\mathcal{D}^{\mathbf{l}_n}_\mathbb{A}$ by
$$\mathcal{D}^{\mathbf{l}_n}_\mathbb{A}(s):=\sum_{n=1}^\infty\frac{\mathbf{l}_n}{a_n^s}$$
If $\mathbf{l}_n=1$ we write
$\mathcal{D}_\mathbb{A}(s):=\sum_{a\in\mathbb{A}}\frac{1}{a^s}$
with respect to the natural order on $\mathbb{A}$ induced by the order of $\mathbb{N}$ and whenever this Dirichlet series converges.

If $1\notin\mathbb{A}$ we denote by $\zeta^{\mathbf{l}_n}_{\mathbb{A}}$ whenever it makes sense the Euler product
$$\zeta^{\mathbf{l}_n}_{\mathbb{A}}(s):=\prod_{n=1}^\infty\frac{1}{1-{\mathbf{l}_n}a_n^{-s}}$$
If $\mathbf{l}_n=1$ we again use the shorthand notation
$\zeta_\mathbb{A}(s):=\prod_{n=1}^\infty\frac{1}{1-a_n^{-s}}$.
\end{definition}
The $\exp$ function satisfies $\exp(s)\neq0\forall{s}\in\mathbb{C}\setminus\infty$ and connects by
$\exp(x+y)=\exp(x)\exp(y)$
addition and multiplication and \ref{Universality} implies the continuation
\begin{align*}
\exp\bigr(\mathcal{D}_\mathbb{P}(s)\bigr)&=\exp\Biggr(\sum_{p\in\mathbb{P}}p^{-s}\Biggr)=\prod_{p\in\mathbb{P}}\exp\bigr({p^{-s}}\bigr)=\zeta(s)\prod_{p\in\mathbb{P}}\exp(p^{-s})(1-p^{-s})
\end{align*}
for $\Re(s)>1/2$, where the r.h.s. of this equation gives an analytic continuation if we think of
$\zeta=\mathcal{D}_\mathbb{N}=\zeta_{\mathbb{P}}$
as determined by Riemann's original analytic continuation \cite{RI}.

Hence the zeros and poles of $\zeta=\mathcal{D}_\mathbb{N}(s)$ correspond to logarithmic singularities of the Dirichlet series
$\mathcal{D}_\mathbb{P}(s)=\sum_{p\in\mathbb{P}}1/p^{s}$
known as the prime $\zeta$ function. This reformulation of zeros and poles of $\zeta$ to logarithmic singularities of $\mathcal{D}_\mathbb{P}$ is well-known because for $\Re(s)>0$ by M\"obius inversion
the identities
$$\mathcal{D}_\mathbb{P}(s)=\sum_{n>0}\frac{\mu(n)\ln\bigr(\zeta(ns)\bigr)}{n}$$
$$\ln\bigr(\zeta(s)\bigr)=\sum_{n>0}\frac{{\mathcal{D}_\mathbb{P}}(ns)}{n}$$
hold where $\mu$ is the M\"obius function, but the auxiliary lemma \ref{BIGPRO} seems easier for some purposes.

Considering
$\prod_{n=1}^\infty\exp\left(\mathbf{l}_na_{n}^{-s}\right)=\exp\left(\mathcal{D}^{\mathbf{l}_n}_{\mathbb{A}}(s)\right)$
and then
$\Gamma(s)\sum_{n=1}^\infty\mathbf{l}_n{a}_{n}^{-s}$
we yield under the assumption of an in some sense low Dirichlet density non-zero continuations:
\begin{korollar}\label{Kor1}
Let $\lambda\in\mathbb{R}^+$ and $\{a_n\}_{n\in\mathbb{N}^+}$ be a sequence with $a_{i+1}>a_i\in\mathbb{N}^+$. If
$\sum_{n=1}^\infty{e}^{-\mathbf{l}_na_{n}t}$
is majorized by $t^{-\lambda}$ for $t\rightarrow0$ then the Euler product
$\prod_{n=1}^\infty{1}/({1-\mathbf{l}_na_{n}^{-s}})$
admits a continuation to $\Re(s)>\max(1/2,\lambda)$ with no zeros or singularities.
\end{korollar}
 It is also natural to consider for example the deformation
$\mathcal{D}^z_{\mathbb{P}}(s):=\sum_{p\in\mathbb{P}}z^p/p^s$.
Clearly we have the identity
$\mathcal{D}^1_{\mathbb{P}}(s)=\mathcal{D}_{\mathbb{P}}(s)$
and moreover this deformation converges as usual for $\vert{z}\vert<1$ and $\Re(s)>0$ majorized by the polylogarithm function.
\section{\textit{At $\infty$ alternating Euler products and Leibniz division}}\label{Alternatingsection}
Consider the set $\mathbb{P}$ of primes endowed with the natural order $p_i\leq{p}_j$ if $i\leq{j}$, for instance
we set
$$\{p_1,p_2,p_3,p_4,p_5,p_6,p_7,p_8,p_{9},\dots\}=\{2,3,5,7,11,13,17,19,23,\dots\}$$
\begin{proposition}\label{Approximation}
Let $N\in\mathbb{N}^+$ be a natural number and consider a sequence of numbers $\{\mathbf{l}^N_n\}\in[1,-1]$ with $\mathbf{l}^N_n=\pm(-1)^n$ if $N\leq{n}$ and a sequence of integers $\mathbb{A}=\{a_n\}_{n\in\mathbb{N}^+}$ with $2\leq{a_1}<\cdots<{a}_n<{a}_{n+1}\in\mathbb{N}^+$. The Euler product
$\zeta^{\mathbf{l}^N_n}_{\mathbb{A}}(s)$
has an analytic continuation to $\Re(s)>1/2$ and no zeros or singularities in this domain.
\end{proposition}
\begin{proof}[Proof]
First we consider the function
$\exp\left(\mathcal{D}^{\mathbf{l}_n^N}_{\mathbb{A}}(s)\right)$. A finite amount of changes in a sum does not change its convergence. By Leibniz convergence argument for alternating series it is clear that the series $\sum_{n=1}^\infty{\mathbf{l}_n^N}{a_n^{-s}}$ converges for every real $s\in(0,\infty]$ and by Abel summation it is well-known that this implies that it converges for $\Re(s)>0$. By \ref{Universality} this implies for $\Re(s)>1/2$ that the Euler product $\zeta^{\mathbf{l}^N_n}_{\mathbb{A}}$ has no zeros or poles in the half-plane $\Re(s)>1/2$.\end{proof}
\begin{proposition}\label{Induction}
Let $\mathbb{A}=\{a_n\}_{n\in\mathbb{N}^+}$ be a set of integers $\geq2$ and $\mathbb{A}'=\{a'_n\}_{n\in\mathbb{N}^+}\subset\mathbb{A}$ be a subset with $\sum_{n=1}^\infty{e}^{-a'_{n}t}$ majorized by $t^{-1/2}$ for $t\rightarrow0$. Consider $\tilde{\mathbb{A}}:=\mathbb{A}\setminus\mathbb{A}'=\{\tilde{a}_n\}_{n\in\mathbb{N}^+}$.
If
$\mathcal{D}_{\{\tilde{a}_{2^i{n}+j}\}_{n\in\mathbb{N}^+}}(s)$
for some $i,j\in\mathbb{N}$ admits a non-singular analytic continuation to some domain
$\mathcal{U}$ then we also have
$\zeta_{\mathbb{A}}(s)\neq0\forall{s}\in\mathcal{U}\cap\{s\vert\Re(s)>1/2\}$.
\end{proposition}
\begin{proof}[Proof] First we can use corollary \ref{Kor1} to exclude a subset $\mathbb{A}'\subset\mathbb{A}$ with low Dirichlet density, consider the natural order on 
$\tilde{\mathbb{A}}$ and proceed in the following way:

By \ref{Approximation}
$\exp\left(\sum_{n=1}^\infty{(-1)^n}/{\tilde{a}_n^s}\right)\neq0,\infty$
for $\Re(s)>0$ and if $\Re(s)>1$ this function equals
\vspace{-0.17cm}$$\vspace{-0.17cm}{\exp\Biggr(\sum_{n=1}^\infty{\tilde{a}_{2n}^{-s}}\Biggr)}\Bigg/{\exp\Biggr(\sum_{n=0}^\infty{\tilde{a}_{2n+1}^{-s}}\Biggr)}$$
If $\exp\bigr(\sum_{n=1}^\infty{\tilde{a}_{2n}^{-s}}\bigr)$ or $\exp\bigr(\sum_{n=0}^\infty{\tilde{a}_{2n+1}^{-s}}\bigr)$ admits a non-singular continuation except at $s=1$ and with no zero for $\Re(s)>\frac{1}{2}$ we can conclude by the well-behaved quotient that both functions admit continuations for $\Re(s)>{1}/{2}$ without singularities or zeros for $s\neq1,\Re(s)>{1}/{2}$, hence
$$\exp\bigr(\mathcal{D}_{\tilde{\mathbb{A}}}(s)\bigr)=\exp\left(\sum_{n=1}^\infty{\tilde{a}_{2n}^{-s}}\right)\exp\left(\sum_{n=1}^\infty{\tilde{a}_{2n-1}^{-s}}\right)\neq0,\infty$$
for $s\neq1,\Re(s)>\frac{1}{2}$. Now continue to split the appearing products inductively.
\end{proof}
\begin{korollar}\label{PrimeSplit}
Let $\tilde{\mathbb{P}}:=\mathbb{P}\setminus\mathbb{P}'$ where $\mathbb{P}'\subset\mathbb{P}$ is a subset of the primes with $\sum_{n=1}^\infty{e}^{-p'_{n}t}$ majorized by $t^{-1/2}$ for $t\rightarrow0$.
 If there exist $i,j\in\mathbb{N}$ so that
$\mathcal{D}_{\{\tilde{p}_{2^i{n}+j}\}_{n\in\mathbb{N}}}(s)$
admits a non-singular continuation to some domain
$\mathcal{U}$ then we also have
$\zeta(s)\neq0\forall{s}\in\mathcal{U}\cap\{s\vert\Re(s)>1/2\}$.
\end{korollar}
\begin{proof}[Proof] The corollary \ref{PrimeSplit} is obviously a direct consequence of the previous proposition \ref{Induction}. For example if the function $\exp\bigr(\sum_{n=1}^\infty{\tilde{p}_{2n}^{-s}}\bigr)$ admits a non-singular analytic continuation except at $s=1$ and with no zero for $\Re(s)>{1}/{2}$ we just set
\vspace{-0.17cm}\begin{equation}\label{Split}
\vspace{-0.17cm}\exp\bigr(\mathcal{D}_{\tilde{\mathbb{P}}}(s)\bigr)=\exp\Biggr(-\sum_{n=1}^\infty{(-1)^n}{\tilde{p}_n^{-s}}\Biggr)\exp^2\Biggr(\sum_{n=1}^\infty{\tilde{p}_{2n}^{-s}}\Biggr)
\end{equation}
Graphically the splitting of order two of the original Euler product $\zeta$ is represented by the following rooted perfect binary tree
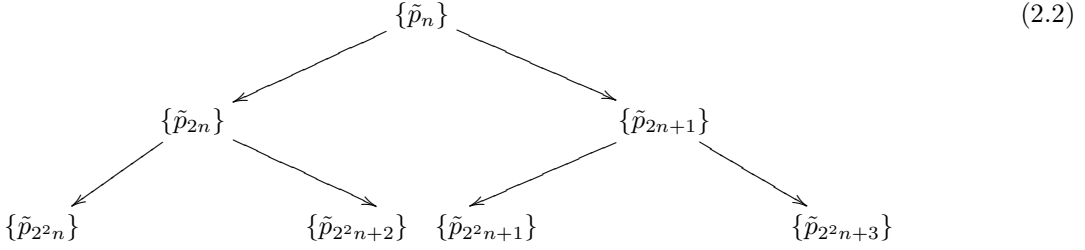
\hspace{-0.8cm}\begin{align}\label{tree}
\xymatrix{&&\{\tilde{p}_{n}\}\ar[dl]\ar[dr]&&&\\
&\{\tilde{p}_{2n}\}\ar[dl]\ar[dr]&&\{\tilde{p}_{2n+1}\}\ar[dl]\ar[dr]&&\\
\{\tilde{p}_{2^2n}\}&&\{\tilde{p}_{2^2n+2}\}\quad\{\tilde{p}_{2^2n+1}\}&&\{\tilde{p}_{2^2n+3}\}\\
}\end{align}
If for example $\exp\bigr(\sum_{n=1}^\infty{\tilde{p}_{2^2n}^{-s}}\bigr)$ admits an analytic continuation we first continue
$\exp\bigr(\sum_{n=1}^\infty{\tilde{p}_{2n}^{-s}}\bigr)$ by
$\exp\left(-\sum_{n=1}^\infty{(-1)^n}{\tilde{p}_{2n}^{-s}}\right)\exp^2\left(\sum_{n=1}^\infty{\tilde{p}_{2^2n}^{-s}}\right)$
and then again $\exp\bigr(\mathcal{D}_{\tilde{\mathbb{P}}}(s)\bigr)$ with formula \ref{Split}.\end{proof}
In fact under the assumption that one of the leaves of the tree \ref{tree} admits a continuation we can continue by Leibniz convergence argument also all other leaves analogous to \ref{Split}.  Notice the Leibniz division
$A\bullet{B}:=A/B$
of two Euler products $A=\exp\left(\mathcal{D}_{\{\tilde{p}_{2^i{n}+j}\}_{n\in\mathbb{N}}}(s)\right)$ and $B=\exp\left(\mathcal{D}_{\{\tilde{p}_{2^i{n}+j'}\}_{n\in\mathbb{N}}}(s)\right)$ results in
$$A\bullet{B}=\exp\left(\mathcal{D}_{\{\tilde{p}_{2^i{n}+j}\}_{n\in\mathbb{N}}}(s)-\mathcal{D}_{\{\tilde{p}_{2^i{n}+j'}\}_{n\in\mathbb{N}}}(s)\right)$$
hence a for $\Re(s)>0$ well-defined continuation, but this binary operation $\bullet$ is for example only linear in $A$, non-commutative
$A\bullet{B}=(B\bullet{A})^{-1}$
and non-associative
$$A\bullet(B\bullet{C})-(A\bullet{B})\bullet{C}=\frac{A}{B}\left(C-\frac{1}{C}\right)$$
Also for the skew-symmetrization
$[A,B]_-^\bullet:=A\bullet{B}\pm{B}\bullet{A}=A/B\pm{B}/A$
we have a Jacobi identity defect
$$[[A,B]^\bullet_-,C]^\bullet_-+[[B,C]^\bullet_-,A]^\bullet_-+[[C,A]^\bullet_-,B]^\bullet_-=ABC\Bigr/\left(\frac{-1}{A^2-B^2}+\frac{-1}{B^2-C^2}+\frac{-1}{C^2-A^2}\right)$$
Notice the previous expression unfortunately does not serve as a justified well-defined analytic continuation for $\Re(s)>0$ by the Leibniz division argument: The domain where $\bullet$ extends to $\Re(s)>0$ in the sense of \ref{Approximation} consists of pairs $\left(\mathcal{D}_\mathbb{A},\mathcal{D}_\mathbb{B}\right)$ of positive and at $\infty$ interlacing series in the sense that there exists $i,j\in\mathbb{N}$ with
$a_i<b_j<a_{i+1}<b_{j+1}<a_{j+2}<b_{j+2}\cdots$.

It might be far stretched but seems also a bit suspicious that the number of rooted perfect binary trees with $n$ leaves is well-known to be given by the Catalan numbers
$C_{n}=\frac{1}{n}\binom{2n-2}{n-1}=\frac{1}{n!}\frac{\partial^{n}}{\partial{z}^{n}}\frac{1-\sqrt{1-4z}}{2}\Big\vert_0\;\forall\;{n}\in\mathbb{N}^+$
appearing in many combinatorial problems.\\

Let again $\tilde{\mathbb{P}}:=\mathbb{P}\setminus\mathbb{P}'$ where $\mathbb{P}'\subset\mathbb{P}$ is a subset of the primes with $\sum_{n=1}^\infty{e}^{-p'_{n}t}$ majorized by $t^{-1/2}$ for $t\rightarrow0$. If there exist $i,j\in\mathbb{N}$ so that
$$\prod_{m=1}^\infty\frac{1}{1-\tilde{p}^{-s}_{2^{a}{m}+b}}$$
admits an analytic continuation to $\Re(s)>1/2$ with no singularities for $s\neq1$ the order of every non trivial zero of $\zeta$ with real part $>1/2$ is at least $i$.
The way in which we continue the Euler products implies this statement, because it is well-known that the vanishing order of $\prod_{m=1}^\infty1/(1-\tilde{p}^{-s}_{2^{i}{m}+j})$ would be an integer $\geq1$ if we assume the existence of an analytic continuation.

There has been a lot of research on the continuation of Taylor and Dirichlet series, some of the established classical theorems are Hadamard's gap theorem and P\'olya's theorem relying on it. Also the more recent work of Bhowmik, Essouabri, Lichtin  Matsumoto and Schlage-Puchta and related results on the natural boundaries of Dirichlet series are interesting in this context, we refer the reader for further considerations to the articles \cite{BSP},\cite{BMa} and \cite{BEL}.\\

Consider the sequences where $\{\mathbf{l}^N_n\}=1$ if $1\leq{n}\leq{N-1}$ and $\mathbf{l}^N_n=(-1)^n$ if $N\leq{n}$. Unfortunately only for $\Re(s)>1$ it is easy to verify uniform convergence of $\zeta^{\mathbf{l}^N_n}_{\mathbb{P}}$ for $N\rightarrow\infty$:

We clearly have
$$(\zeta-\zeta^{\mathbf{l}^N_n})_{\mathbb{P}}(s)=\sum_{n=p_N}^\infty{a_n}/n^s$$
with $a_n\in\{0,2\}$ and can estimate this sum as usual by an integral:
$$\left(\zeta-\zeta^{\mathbf{l}^N_n}_{\mathbb{P}}\right)(s)\leq\sum_{n=p_N}^\infty\frac{2}{n^{\Re(s)}}$$
$$\leq2\int_{p_N}^\infty\frac{\\\d{t}}{t^{\Re(s)}}$$
$$=2\frac{t^{1-\Re(s)}}{1-\Re(s)}\Big\vert^\infty_{p_N}$$
$$=\frac{2{p_N}^{1-\Re(s)}}{\Re(s)-1}$$
It seems reasonable that $\zeta^{\mathbf{l}^N_n}_{\mathbb{P}}$ with $\mathbf{l}^N_n\in\{1,-1\}$ also satisfy the Voronin universality.
\section{The connection of the prime $\zeta$ function and Goldbach-Waring like representations}
The Goldbach conjecture states that all even natural numbers $n$ can be represented as the sum of two primes and more general for integers $i\geq1,m\geq2$ it is interesting to consider so-called \textit{\textbf{Goldbach-Waring representations}}
$${p_1^i+\cdots+{p}_m^i}=n$$
where $p_1,\cdots,p_m$ are prime numbers.

It is not very surprising that the prime $\zeta$ function
$\mathcal{D}_\mathbb{P}(s)$
is related to the Goldbach-Waring representations, in fact it in some sense generates all of them by considering certain Mellin transforms:

Let $1>\nu\in\mathbb{R}$ be a real number and set
$\tilde{\mathbb{P}}:=\mathbb{P}\setminus\mathbb{P}'$
where $\mathbb{P}'\subset\mathbb{P}$ is a subset of the primes with $\sum_{n=1}^\infty{e}^{-p'_{n}t}$ majorized by $t^{-\nu}$ for $t\rightarrow0$. For $\Re(s)>1$ we have by majorized convergence as usual the identity
$$\Gamma(s/k)\mathcal{D}_{\{\tilde{p}_{2^i{n}+j}\}_{n\in\mathbb{N}}}(s)=\int_{0}^\infty\frac{\\\d{t}}{t}t^{{s}/{k}}g_k\bigr(\exp(-t)\bigr)$$
where the functions $g_k:\mathbb{R}^+\rightarrow\mathbb{R}^+$ are defined by
$$g_k(x):=\sum_{\tilde{p}\in\{\tilde{p}_{2^i{n}+j}\}_{n\in\mathbb{N}}}{x}^{\tilde{p}^k}$$
We also have the quite obvious calculation
\begin{align*}
g_k^{m}(x)&=\hspace{-3cm}\sum_{\hspace{3cm}\bigr(\tilde{p}_{2^i{n_1}+j},\cdots,\tilde{p}_{2^i{n_m}+j}\bigr)\in\{\tilde{p}_{2^i{n}+j}\}_{n\in\mathbb{N}}^{m}}\hspace{-3cm}{x}^{\tilde{p}_{2^i{n_1}+j}^k+\cdots+\tilde{p}_{2^i{n_m}+j}^k}\\&=\sum^\infty_{n=2^{i}m}x^n\sum_{{\tilde{p}_{2^i{n_1}+j}^k+\cdots+\tilde{p}_{2^i{n_m}+j}^k}=n}1
\end{align*}
The pole of $\zeta$ at $1$ indicates that there have to be in some sense quite many Goldbach-Waring like representations:

Suppose that for some $m$ the number of Goldbach-Waring like representations
$${\sum_{{\tilde{p}_{2^i{n_1}+j}^k+\cdots+\tilde{p}_{2^i{n_m}+j}^k}=n}1}$$
is bounded by some $B$ for all $n$ then the previous equation implies that $g_k(x)$ would be majorized by $(1-x)^{1/m}$. If we in addition to a bounded number of Goldbach-Waring like representations assume that this holds for some $m>k$ we could continue $\zeta$ without singularities to $\Re(s)>k/m$, but this contradicts the pole of $\zeta$ at $1$, hence there can be no bound $B$ for the Goldbach-Waring like representations if we suppose $m>k$.

It seems reasonable that the previous statement on Goldbach-Waring like representation bounds did already appear in the literature and maybe the statement has been already advanced by other methods or can be deduced by heuristic arguments. There seems to be also an analogy of this estimations with some results of Korevaar.

\vfill{\begin{center}
\textcircled{c} Copyright by Johannes L\"offler, 2014, All Rights Reserved
\end{center}}

\end{document}